\documentclass[11pt]{amsart}
\usepackage{amssymb}
\usepackage{mathptmx}

\theoremstyle{plain}
\newtheorem{theorem}{Theorem}
\newtheorem{corollary}{Corollary}

\newtheorem{proposition}{Proposition}

\theoremstyle{definition}

\theoremstyle{remark}

\numberwithin{equation}{section}
\def\K#1#2{\displaystyle{\mathop K\limits_{#1}^{#2}}}
\def\J#1#2{\displaystyle{\mathop \sum\limits_{#1}^{#2}}}
\begin{document}
\title[Polynomial Continued Fractions ]
       {Polynomial Continued Fractions}
\author{D. Bowman}
\address{Mathematics Department\\
       University of Illinois \\
        Champaign-Urbana, Illinois 61820}
\email{bowman@math.uiuc.edu}
\author{J. Mc Laughlin}
\address{Mathematics Department\\
       University of Illinois \\
        Champaign-Urbana, Illinois 61820}
\email{jgmclaug@math.uiuc.edu}
\keywords{Continued Fractions}
\subjclass{Primary:11A55}
\begin{abstract}
Continued fractions whose elements are polynomial
 sequences\\
have been
carefully studied mostly in the cases where the degree of the
numerator polynomial
is less than or equal to two and the degree of the denominator polynomial
is less than or equal to one. Here we study cases of higher degree for both
numerator and denominator polynomials, with particular attention given to
cases in which the degrees are equal. We extend work of Ramanujan on
continued fractions with rational limits and also consider cases where the
limits are irrational.
\end{abstract}

\maketitle


\section{Introduction } \label{S:intro}

A polynomial continued fraction is a continued fraction
  $ K_{n=1}^{\infty}\displaystyle{a_{n}/b_{n}}$
where $a_n$ and $b_n$ are polynomials in $n$. Most well known continued
fractions are of this type. For example the first continued fractions
giving values for $\pi$\,(due to Lord Brouncker, first published in
 ~\cite{W65})  and $e$\,(\cite{E27})  are of this type:
\vspace{3pt}
\begin{equation}\label{E:lb}
\frac{4}{\pi} = 1 +
\K{n=1}{\infty}\displaystyle{\frac{(2n-1)^{2}}{2}} \,,
\end{equation}
\begin{equation}\label{E:eule}
e =
 2 + \cfrac{1}{1 +
\K{n=1}{\infty}\displaystyle{\frac{n}{n+1}}
}\,.
\end{equation}
Here we use the standard notations
\[
\K{n = 1}{N}\displaystyle{\frac{a_{n}}{b_{n}}} \,:=
\cfrac{a_{1}}{b_{1} + \cfrac{a_{2}}{b_{2} + \cfrac{a_{3}}{b_{3} +
\ldots +  \cfrac{a_{N}^{}}{b_{N} }}}}
  \]
\[
\phantom{asd} =  \frac{a_{1}}{b_{1}+}\, \frac{a_{2}}{b_{2}+}\,
\frac{a_{3}}{b_{3}+}\, \dots  \frac{a_{N}}{b_{N}}.
\]
We write $\displaystyle{A_{N}/B_{N}}$ for the above finite continued
fraction written as a rational function of the variables
$a_{1},...,a_{N},b_{1},...,b_{N}$.
By $K_{n = 1}^{\infty}\displaystyle{a_{n}/b_{n}}$
 we  mean the limit
of the sequence \{$\displaystyle{A_{n}/B_{n}}$\} as \,$n$\,goes to infinity,
if the limit exists.

The first systematic treatment of this type of continued fraction
seems to be in Perron~\cite{oP13} where  degrees through two for $a_n$
and degree one for $b_n$\, are studied. Lorentzen
and Waadeland~\cite{LW92} also study these cases in detail and they evaluate
all such continued fractions in terms of hypergeometric series.
There is presently no such systematic treatment for cases of
higher degree in the and examples in the literature are
accordingly scarcer.
 Of particular interest are cases where the degree of
$a_n$ is less than or equal to the degree of $b_n$. These cases
are interesting from a number theoretic standpoint since
the values of the continued fraction can then be approximated
exceptionally well by rationals and irrationality measures may
then be given. When the degrees are equal, the
value may be rational or irrational and certainly the latter when
the first differing coefficient is larger in $b_n$. (Here we count the
first coefficient as the coefficient of the largest degree term.)
Irrationality follows from the criterion given by Tietze,
extending the famous Theorem of Legendre
 (see Perron~\cite{oP13}, pp. 252-253) :
\vspace{5pt}
\newline
\textbf{Tietze's Criterion}:

Let $ \{\,a_{n}\}_{n = 1}^{\infty}$ be a sequence of integers and
 $\{\,b_{n}\}_{n = 1}^{\infty}$
be a sequence of positive integers, with
$\,a_{n} \, \neq \, 0$ \,for any $n$.
If there exists a positive integer  $N_{0}$  such that
\begin{equation}\label{E:Ttze1}
\begin{cases}
 b_{n} \geq |a_{n}|\, & \\
 b_{n} \geq |a_{n}|\,+ 1, & for\,\, a_{n+1}\,<\,0,
\end{cases}
\end{equation}
for all \,$n\, \geq N_{0}$  then \,
$K_{n = 1}^{\infty}\displaystyle{a_{n}/b_{n}}\,$
converges and it's limit is irrational.

It would seem from the literature that finding cases of equal
degrees or even close degrees is difficult. If one picks a
typical continued fraction from published tables, the degree
of the numerator tends to be twice that of the denominator.
One easy way in which this can arise is when the continued
fraction is equal to a series after using the Euler transformation:
\begin{equation}\label{E:eul}
\sum_{n \geq 0}a_{n}
=
 a_{0} + \frac{a_{1}}{1+}\, \frac{-a_{2}}{a_{1}+a_{2}+}\,
\frac{-a_{1}a_{3}}{a_{2}+a_{3}+}\,\frac{-a_{2}a_{4}}{a_{3}+a_{4}+}\,
\frac{-a_{3}a_{5}}{a_{4}+a_{5}+\dots}\,.
\end{equation}

If one side of this equality converges, then so does the other as the
$n$th approximants are equal. The Euler transformation is easily proved by
induction.

In this formula, if the terms of the series are rational functions
of the index of fixed degree, then in the continued fraction
after simplification, one will get the degrees of the numerators
to be at least twice that of the denominators. The
continued fraction for $\pi$ given by~\eqref{E:lb} is an example of this
phenomenon. Another example of this is the series definition of
Catalan's constant:
\[
C =
\sum_{k=0}^\infty
\frac{(-1)^{k}}{(2k+1)^{2}}
\]
which, by \eqref{E:eul},  transforms into the continued fraction given by
\[
C =  \frac{1}{1+}\,\frac{1^{4}}{8+}\,\frac{3^{4}}{16+}\,
\frac{5^{4}}{24+\ldots}\,.
\]
Here the degree of the numerator is four times that of the
denominator. This continued fraction appears to be new.

Taking contractions of continued fractions (see, for example,
Jones and Thron \cite{JT80}, pp. 38--43)  also leads to a
relative increase in the degree of the
numerator over that of the denominator. For example, forming the even
part of the continued fraction will cause a continued fraction with
equal degrees to be transformed
into one with twice the degree in the numerator as the denominator.\\
The even part of the continued fraction
$K_{n =1}^{\infty}\displaystyle{a_{n}/b_{n}}$
is equal to
\begin{multline*}
\frac{a_{1}b_{2}}{a_{2} + b_{1}b_{2}+ }\,
\frac{-a_{2}a_{3}b_{4}}{a_{3}b_{4} + b_{2}(a_{4}+b_{3}b_{4})  +}\\
\frac{-a_{4}a_{5}b_{2}b_{6}}
{a_{5}b_{6} + b_{4}(a_{6} + b_{5}a_{6}) +}\,
\frac{-a_{6}a_{7}b_{4}b_{8}}
{a_{7}b_{8} + b_{6}(a_{8} + b_{7}a_{8}) +\ldots} .
\end{multline*}

Other work on polynomial continued fractions was done by
Ramanujan~\cite{B89}, chapter 12.
He gave several cases of equal degree in which the sum is rational.
For example, Ramanujan gave the following:
If $x$ is not a negative integer then
\begin{equation}\label{E:Ram1}
\K{n=1}{\infty}\displaystyle{\frac{x+n}{x+n-1}} = 1.
\end{equation}
Despite the simplicity of this formula, Ramanujan did not give
a proof: the first proof seems to be have been given by Berndt
 \cite{B89}, Page 112.

In this paper we examine a large number of infinite classes of
polynomial continued fractions in which the degrees are equal,
or close. Our results follow from a theorem of Pincherle and a variant
of the Euler transformation discussed above. We obtain generalizations
of Ramanujan's results in which
the degrees are equal and the values rational as well as cases
of equal degree with irrational limits. Many of our
theorems give infinite families of continued fractions. While we
concentrate on polynomial continued fractions, many of the results
hold in more general cases.
Here are some special cases of our general results (proofs are given
throughout the paper) :

\vspace{2pt}


{\allowdisplaybreaks
\begin{equation}\label{Ex:eex1}
 \K{n=1}{\infty}\frac{n^{\alpha} + 1}{n^{\alpha}}  = 1,\,\, \text{for}\,\,\,
 \alpha >0.
\end{equation}
\begin{equation}\label{Ex:eex2}
 1 - \frac{\displaystyle{1}}
          {\displaystyle{1} + \K{n =1}{\infty}\displaystyle
          {\frac{n^{2} }
                {n^{2} + 2n }}} = J_{0}(2) ,
\end{equation}
where $J_{0}(x)$ is the Bessel function of the first kind
of order $0$.
\begin{equation}\label{Ex:eex3}
 2 + \K{n =1}{\infty}\displaystyle
{\frac{2n^{2} + n }
      {2n^{2} + 5n + 2 }} =\frac{1}
   {\sqrt{2}\,\, \csc{(\sqrt{2})} - 1 }.
\end{equation}
(Notice that the irrationality criterion mentioned above
means that the last two quantities on the right  are irrational)
\begin{equation}\label{Ex:eex8}
  \K{n = 1}{\infty}\frac{n^{12} + 2n^{11} + n^{10} + 4n + 5}
                      {n^{12} + 4n -4} = 4.
 \end{equation}
\begin{equation}\label{Ex:eex10}
  \K{n =2}{\infty}
   \frac{6n^{7} + 6n^{6} + 2n^{5} + 3n + 2}
        {6n^{7} - 6n^{6} + 2n^{5} + 3n - 5} = \frac{19}{7}.
 \end{equation}
\begin{equation}\label{Ex:eex13}
  \K{n =1}{\infty}
   \frac{3n^{6} - 3n^{4}+n^{3}+6n^{2}+6n-1 }
         {3n^{6} - 9n^{5} + 6n^{4}+n^{3}+6n^{2}-12n-2 } = 1.
   \end{equation}
}

\section{Infinite Polynomial Continued Fractions with Rational Limits}

In this section we derive some general results about the convergence of
polynomial continued fractions in some infinite families and give some
examples of how these results can be used to find the limit of such
continued fractions. Many of the results in this paper are consequences of
the following theorem of Pincherle ~\cite{P94} :


\begin{theorem}\label{T:P*}
(Pincherle)  Let $ \{\,a_{n}\}_{n = 1}^{\infty}$,
$\{\,b_{n}\}_{n = 1}^{\infty}$ and
 $ \{\,G_{n}\}_{n = -1}^{\infty}$ be \\ sequences of real or complex
  numbers satisfying  $a_{n} \neq 0$ for $n\geq 1$
 and for all \,\, $n \geq 1$,
\begin{equation}\label{E:pn}
 G_{n} =   a_{n}G_{n-2}  + b_{n} G_{n - 1}.
\end{equation}
Let $\{B_{n}\}_{n=1}^{\infty}$
denote the denominator convergents of the continued fraction
 $\K{n = 1}{\infty}\displaystyle{\frac{a_{n}}{b_{n}}}$.

If $\lim_{n \to \infty}\displaystyle{G_{n}/B_{n}} = 0$
then  $\K{n = 1}{\infty}\displaystyle{\frac{a_{n}}{b_{n}}}$
converges and  its limit is $-G_{0}/G_{-1}$.
\end{theorem}

\begin{proof}
See, for example, Lorentzen and Waadeland~\cite{LW92}, page 202.
\end{proof}

For many sequences it may be difficult to decide
whether the
 condition
 $\lim_{n \to
\infty}$
$\displaystyle{G_{n}/B_{n}} = 0$\, is
satisfied. Below are some easily proven properties governing the growth of the
$B_{n}$'s which will be useful later.

\vspace{15pt}
(i) \quad
Let $a_{n}$\, and \, $b_{n}$
be  non-constant polynomials in $n$ such that $a_{n} \geq 1$,
$b_{n} \geq 1$,
for $n\geq1$ and suppose $b_{n}$ is a polynomial of degree\,$k$.
If the leading coefficient of $b_{n}$ is $D$, then given $\epsilon >0$,
there exists
a positive constant $C_{1}=C_{1}(\epsilon) $ such that
$B_{n} \geq C_{1}(|D|/(1+\epsilon))^{n}(n!)^{k}$.

\vspace{5pt}
(ii) \quad
If $a_{n}$\, and \, $b_{n}$ are positive numbers $\geq 1$, then  there
exists a
positive
constant $C_{3}$ such that $B_{n} \geq C_{3}\phi^{n}$ for $n\geq 1$,
where $\phi$ is
the golden ratio\, $\displaystyle{(1 + \sqrt{5})/2}$.
\vspace{10pt}

\begin{corollary}\label{C:c3}
If $m$ is a positive integer and
$b_{n}$ is any polynomial of degree $\geq1$ such that $b_{n} \geq 1$
for $n\geq1$, then
\[
 \K{n =1}{\infty}\displaystyle{\frac{mb_{n} + m^{2}}{b_{n}}} = m.
\]
\end{corollary}

\begin{proof}
With $a_{n} = mb_{n} +
m^{2}$, for \,$ n \geq \, 1$, and  $G_{n} = (-1)^{n+1}m^{n+1}$,  for
\,$n\geq -1$,  $ \{\,a_{n}\}_{n = 1}^{\infty}$,
$\{\,b_{n}\}_{n = 1}^{\infty}$ and
 $ \{\,G_{n}\}_{n = -1}^{\infty}$ satisfy equation~\eqref{E:pn}.
By (i)
above
\[
\lim_{n \to
\infty}\displaystyle{G_{n}/B_{n}} =
\lim_{n \to
\infty}
\displaystyle{(-1)^{n+1}m^{n + 1}/B_{n}} = 0 \Longrightarrow
\K{n =1}{\infty}\displaystyle{\frac{a_{n}}{b_{n}}} = -G_{0}/G_{-1} = m.
\]
\end{proof}

A special case is where $m = 1$, in which case $|G_{n}| = 1$, for all
$n$ and all that is necessary is that $\lim_{n \to \infty} B_{n} =
\infty$. The following generalization
of  the result \eqref{E:Ram1} of Ramanujan, for positive numbers
greater  than 1 follows easily:

If $ \{\,b_{n}\}_{n =1}^{\infty}$ is any sequence of positive numbers
with $b_{n} \geq 1$ for $n\geq 1$ then
\[
\K{n = 1}{\infty}\displaystyle{\frac{b_{n} + 1}{b_{n}}} = 1.
\]

Letting $b_{n} = n^{\alpha}$, $\alpha >0$, gives \eqref{Ex:eex1} in the
introduction.\footnote{Lorentzen
and Waadeland give an exercise \cite[page 234]{LW92},
question 15(d)  which effectively involves a similar result in the case where
$b_{n}$ belongs to a certain family of quadratic polynomials in $n$
over the complex numbers.}

\vspace{10pt}

Entry 12 from the chapter on continued fractions in
Ramanujan's
second notebook \cite{B89}\,, page 118, follows  as a consequence of the
above theorem:


\begin{corollary}\label{C:Ram1}
If $x$ and $a$ are complex numbers, where $a \ne 0$ and $x \ne -ka$, where $k$ is a positive integer, then
\[
\displaystyle{ \frac{x+a}{a + \K{n =1}{\infty}
\displaystyle{\frac{(x + na)^{2} - a^{2}}{a}}}} = 1.
\]
\end{corollary}

\begin{proof}
Note that\,
\[\displaystyle{ \frac{x+a}{a + \K{n =1}{\infty}
\displaystyle{\frac{(x + na)^{2} - a^{2}}{a}}}} \, = \,
\displaystyle{ \frac{\displaystyle{\frac{x}{a} + 1}}
{1 + \K{n =1}{\infty}
\displaystyle
{\frac{\left(\displaystyle{\frac{x}{a}} + n\right)^{2} - 1}{1}} }}.
\]
Replace  $\displaystyle{x/a}$  by $m$ to simplify notation;
 the
result will follow if it can be shown that
\begin{equation*}
m =\displaystyle{  \K{n =1}{\infty}
\displaystyle{\frac{(m + n)^{2} - 1}{1}}}
=\displaystyle{  \K{n =1}{\infty}
\displaystyle{\frac{(m + n - 1)(m + n +1)}{1}}}.
\end{equation*}
With \,$G_{-1} = 1$,  $G_{n} = (-1)^{n+1}\prod_{i=0}^{n}(m+i)$,
for \,$n\geq 0$
 and   $b_{n} = 1$, $a_{n} = (m + n - 1)(m + n +1)$ for
 \,$ n \geq \, 1$, $ \{\,a_{n}\}_{n = 1}^{\infty}$,
$\{\,b_{n}\}_{n = 1}^{\infty}$ and
 $ \{\,G_{n}\}_{n = -1}^{\infty}$ satisfy equation~\eqref{E:pn}
so that  the result
will follow from Theorem~\ref{T:P*} if it can be shown that
$\lim_{n \to
 \infty}\displaystyle{G_{n}/B_{n}} = 0$, in which case
the continued fraction will converge to $-G_{0}/G_{-1} = m$.
However,
an easy induction  shows that for
$k \geq \, 1$,
\begin{align*}
&B_{2k + 1} = (k + 1)\prod_{i=2}^{2k+1}(m+i),\quad
\text{and}\,\quad
B_{2k} =(m + k + 1)\prod_{i=2}^{2k}(m+i).
\end{align*}
Thus $\lim_{n \to
 \infty}\displaystyle{G_{n}/B_{n}} = 0$, and the result
follows.
\end{proof}

\vspace{10pt}


\begin{corollary}\label{C:c5}
Let $m$ be a positive integer and let
 $b_{n}$ be any polynomial of degree $\geq 1$ such that
$b_{n} \geq 1,$  for $n\geq1$ and either degree $b_{n} > 1$ or
if  degree \,$b_{n} = 1$
then its leading coefficient is $D > m$.
Then
\[\K{n =
1}{\infty}\displaystyle{\frac{mnb_{n} + m^{2}n(n +1)}{b_{n}}} = m.
\]
\end{corollary}

\begin{proof}
Letting $G_{n} = (-1)^{n+1}m^{n+1}(n + 1)!$ for\, $n\geq -1$ and
$a_{n} = mnb_{n} + m^{2}n(n +1)$, for \,$ n \geq \, 1$, one has that
$ \{\,a_{n}\}_{n = 1}^{\infty}$,
$\{\,b_{n}\}_{n = 1}^{\infty}$ and
 $ \{\,G_{n}\}_{n = -1}^{\infty}$ satisfy equation~\eqref{E:pn}.
By (i),
$\lim_{n \to
\infty}\displaystyle{G_{n}/B_{n}}
 = 0 \Longrightarrow
\K{n =1}{\infty} \displaystyle{ \frac{a_{n}}{b_{n}}} = -G_{0}/G_{-1} = m$.
\end{proof}

\text{Theorem~\ref{T:P*}} does not say directly how to
find the value of all
polynomial continued fractions
 $K_{n =
1}^{\infty}\displaystyle{a_{n}/b_{n}}$ as it does not say
how the sequence ${G_{n}}$ can be found or even if such a sequence can
be found. However, Algorithm Hyper (see~\cite{PWZ96})  can be used to
determine if a hypergeometric solution ${G_{n}}$ exists to
equation~\eqref{E:pn}
and, if such a solution exists, the algorithm will out-put $G_{n}$,
enabling the limit of the
 continued fraction to be found, if ${G_{n}}$ satisfies $\lim_{n \to
 \infty}\displaystyle{G_{n}/B_{n}} = 0$.

Even if for the particular polynomial  sequences ${a_{n}}$ and ${b_{n}}$
it turns out that the sequence ${G_{n}}$ found does not satisfy $\lim_{n \to
 \infty}\displaystyle{G_{n}/B_{n}} = 0$, then these three sequences
${a_{n}}$, ${b_{n}}$ and ${G_{n}}$  may be used to find the value of
 infinitely many other continued fractions when ${G_{n}}$ is a
 polynomial or rational function in $n$.

The following proposition shows how, given any one solution of
(\ref{E:pn}), one can find the value of infinitely many other
polynomial continued fractions in an easy way.
\vspace{10pt}


\begin{proposition}\label{T:hype}
Suppose that there exists complex sequences
$ \{\,G_{n}\,\}_{n =-1}^{\infty}$,\\
$ \{\,a_{n}\,\}_{n = 1}^{\infty}$ and
$ \{\,b_{n}\,\}_{n = 1}^{\infty}$ satisfying
\begin{equation} \label{E:hyp2}
a_{n}G_{n-2} + b_{n}G_{n-1} - G_{n} = 0
\end{equation}
Let $f_{n}$ be any sequence,  let
$s_{n} = f_{n}G_{n-1} + a_{n}$ be such that $s_{n} \ne 0$, for $n \geq 1,$\,
and let \,\,$t_{n} = f_{n}G_{n-2} - b_{n}$ .
Let $A_{n}/B_{n}$ denote the convergents to
$\K{n = 1}{\infty}\displaystyle{\frac{s_{n}}{t_{n}}}$.
If $\lim_{n \to \infty}\displaystyle{G_{n}/B_{n}} = 0$
then  $\K{n = 1}{\infty}\displaystyle{\frac{s_{n}}{t_{n}}}$
converges and  its limit is $G_{0}/G_{-1}$.
\end{proposition}

\begin{proof}
Let $G'_{n} =(-1)^{n+1}G_{n}$. Then
\begin{align*}
&s_{n}G'_{n-2} + t_{n}G'_{n-1} - G'_{n}
= (-1)^{n-1}(a_{n}G_{n-2} + b_{n}G_{n-1} - G_{n})
 \,\, = 0.
\end{align*}

Thus $ \{\,G'_{n}\,\}_{n =-1}^{\infty}$,
$ \{\,s_{n}\,\}_{n = 1}^{\infty}$ and
$ \{\,t_{n}\,\}_{n = 1}^{\infty}$ satisfy the conditions of
theorem~\ref{T:P*}  so $K_{n =
1}^{\infty}\displaystyle{s_{n}/t_{n}}$ converges and its
limit is $-G'_{0}/G'_{-1} = G_{0}/G_{-1}.$\\[5pt]
\end{proof}

\text{Entry 9} from the chapter on continued fractions in Ramanujan's
second notebook~\cite{B89}\,, pages 114-115, follows in the case $a$
is real and positive
and $x$ is real as a consequence of
the above proposition:


\begin{corollary}
Let $a$ be a real positive number and let  $x$ be a real number
 such that  $x \ne -ka$ where
k is a positive integer.
Then
\begin{equation*}
\frac{x+a+1}{x+1} =  \K{n =1}{\infty}
\displaystyle
{\frac{x + na}{x+(n-1)a -1}} .
\end{equation*}
\end{corollary}

\begin{proof}
It is enough to prove this for $x-1 >0$
since for $n$ sufficiently large
 $x+(n-1)a-1>0$ and then the result will hold for
 a tail of the continued fraction and then resulting finite continued
fraction will collapse from the bottom up to give the result.
Let $G_{n} = \displaystyle{(x + (n+1)a+1)/(x+1)}$.
Put $f_{n} = x + 1,\,\, a_{n} = -1$ \,\, and \,\,
$b_{n} =  2$ so that
$a_{n}G_{n-2} + b_{n}G_{n-1} - G_{n} = 0$,  $x + na =
f_{n}G_{n-1}+a_{n}$ and $x+(n-1)a -1 = f_{n}G_{n-2}-b_{n}$.
Since $G_{n}$ is a degree 1 polynomial in $n$, $x+na,x+(n-1)a -1 > 0$
\, for  $n \, \geq \, 1$, it can easily be shown that
$\lim_{n \to \infty}\displaystyle{G_{n}/B_{n}} = 0$
and so by  Proposition~\ref{T:hype} the continued fraction converges to
 $G_{0}/G_{-1} = \displaystyle{(x + a+1)/(x+1)}$.
\end{proof}

\vspace{5pt}

\text{Remarks:}
(1)  In \text{Proposition~\ref{T:hype}} any polynomial $G_{n}$
satisfying   \eqref{E:hyp2}  can always be assumed to have
positive leading coefficient (if necessary multiply
\eqref{E:hyp2} by $-1$.)  If $f_{n}$ is then taken to be a polynomial
of sufficiently high degree
with leading positive coefficient then both $s_{n}$ and $t_{n}$ will
be polynomials with positive leading coefficients so that  there exists a
positive integer $N_{0}$ so that  for all  $n \geq N_{0}$,
$s_{n}$,  $t_{n} > 0$. If it happens that for some $m  \geq N_{0}$
that both $B_{m}$ and $B_{m+1}$ are of the same sign then $B_{n}$
will go to $+ \infty$ or $- \infty$ exponentially fast. In these
circumstances $\lim_{n \to
\infty}\displaystyle{G_{n}/B_{n}} = 0,$ since $G_{n}$
is only of polynomial growth.

In many of the following corollaries $f_{n}$ will be restricted so as to have
$N_{0}$ small (typically in the range $1 \leq  N_{0} \leq 3$), but of
course there are $f_{n}$ for which this is not the case but for which
the results claimed in the corollaries hold.

\vspace{10pt}
(2)  One approach is to take the polynomial $G_{n}$ as given and search for
polynomials $a_{n}$ and $b_{n}$ satisfying equation~\eqref{E:hyp2}.
It can be assumed  that
degree($a_{n}$), degree($b_{n}$)  $<$ degree($G_{n}$). This follows
since
if a solution exists with degree($a_{n}$) $\geq$ degree($G_{n}$) then
 the Euclidean algorithm can be used to write
$a_{n}= p_{n}G_{n-1}+a_{n}'$, $b_{n}= q_{n}G_{n-2}+b_{n}'$,
where $p_{n}$, $q_{n}$, $a_{n}'$ and $b_{n}'$ are polynomials in $n$.
Substituting into \eqref{E:hyp2} and comparing degrees gives that
\eqref{E:hyp2} holds with $a_{n}$ replaced with $a_{n}'$ and
$b_{n}$ replaced with $b_{n}'$.

\vspace{10pt}
(3)  In theory it is  possible to find polynomials $G_{n}$
of arbitrarily high degree and polynomials $a_{n}$ and $b_{n}$ of
lesser degree (with \emph{rational} coefficients)  satisfying
\eqref{E:hyp2}, by using \eqref{E:hyp2} to define equations expressing the
coefficients of $a_{n}$ and $b_{n}$ in terms of those of $G_{n}$.
If $G_{n}$ has degree $k$ and $a_{n}$ and $b_{n}$ both have degree
$k -1$, then \eqref{E:hyp2} is a polynomial identity of degree
$2k - 1$, giving $2k$ equations for the\, $2k$ coefficients of
$a_{n}$ and $b_{n}$.\footnote{Starting with $a_{n}$ and $b_{n}$,
arbitrary polynomials of a certain degree, it is possible to look for
solutions $G_{n}$ satisfying \eqref{E:hyp2}
with coefficients defined in terms of those of $a_{n}$ and $b_{n}$
using the  the Hyper Algorithm (see \cite{PWZ96}). However there is no
certainty that the solutions (if they exist)
will be polynomials or that they will have any particular desired
degree.}

In practice these equations and the requirement that the
coefficients of $G_{n}$ be integers introduces conditions on the
coefficients of $G_{n}$.
For example, if there exists $G_{n} = an^{2} + bn + c$,
 $a_{n} = dn + e$, and $b_{n} = fn + g$,
polynomials with integral coefficients,
satisfying \eqref{E:hyp2}  , then
{\allowdisplaybreaks
\begin{align}
d = -f &= \displaystyle{\frac{4a^{2}}{a^{2} - b^{2} + 4ac }}\,,\notag\\
    e &= \displaystyle{\frac{-3a^{2} + b^{2} +2ab-4ac }
			{a^{2} - b^{2} + 4ac }}\,,\notag\\
    g &= \displaystyle{\frac{12a^{2}-2ab-2b^{2}+8ac}
{a^{2} - b^{2} + 4ac }}\,,\notag
\end{align}
}
giving restrictions on the allowable values of $a$,  $b$ and $c$.
\vspace{2pt}

(Parts (ii) --(ix) of the following corollary correspond, respectively,
to the solutions
$\{a = b = m,c=1\}$,
$\{a = 1, b = 4,c=4\}$,
$\{a = m^{2},b = 3m^{2} - 2m,c=2m^{2} - 2m+1\}$,
$\{a = m^{2}, b = m^{2} + 2m,c=2m+1\}$,
$\{a = m, b = 3m,c=2m+1\}$,
$\{a = m, b = m - 2,c=-1\}$,
$\{a = m, b = 3m + 2,c=2m+3\}$ and
$\{a =  4m, b = 16m^{2} + 8m + 1, c = 16m^{3} + 16m^{2} + 5m + 1\}$  )

Proposition \ref{T:hype} is too general to easily calculate
the limit of particular polynomial continued fractions. The following
corollary enables
these limits to be calculated explicitly in many particular cases.


\begin{corollary}
 Let $m$ be a positive integer, $k$ a positive integer greater than $m$ and
$ \{\,f_{n} \}_{n =1}^{\infty}$  a non-constant polynomial
sequence such that $f_{n} \geq 1$, for $n \geq 1$. For each of
continued fractions below assume that $f_{n}$ is such that no
numerator partial quotient is equal to zero. (This
holds automatically in cases (i)  --(vi)).

\vspace{10pt}

\begin{flushleft}(i)\phantom{asdasdsadfdfdsf}
$\K{n =
1}{\infty}\displaystyle{\frac{(mn +k-m)f_{n}- 1}{(mn +k-2m)f_{n} - 2
}} =
\frac{k}{k-m}$.
\end{flushleft}

\vspace{1pt}

\vspace{1pt}

\begin{flushleft}
(ii)\phantom{asdasdsa}
$\K{n =
1}{\infty}\displaystyle{\frac{((n^{2} - n)m + 1)f_{n} + nm - 1 }
				{((n^{2} - 3n + 2)m + 1)f_{n} + mn -
(2m + 2) }} = 1$.
\end{flushleft}

\vspace{2pt}

\vspace{1pt}

\begin{flushleft}(iii)\phantom{asdasdsadfdfdsf}
$\K{n =1}{\infty}\displaystyle{\frac{ (n + 1)^{2}f_{n} + 4n + 5 }
			{n^{2}f_{n} +4n -4 }} = 4$.
\end{flushleft}

\vspace{10pt}

\vspace{1pt}

\begin{flushleft}
(iv) $\K{n =1}{\infty}
\displaystyle{\frac{f_{n}(m^{2}n^{2} + n(m^{2} - 2m) + 1) + mn + m - 2 }
{f_{n}(m^{2}n^{2} - n(m^{2} + 2m) + 2m + 1) + mn - m -3  }}
= 2m^{2} - 2m+ 1$.
\end{flushleft}

\vspace{10pt}

\vspace{5pt}

\small{
\begin{flushleft}
(v) $\K{n =1}{\infty}
\displaystyle{
\frac{f_{n}(m^{2}n^{2} + n(2m - m^{2}) + 1) + mn }
{f_{n}(m^{2}n^{2} - n(3m^{2} - 2m) + + 2m^{2} + 2m + 1) + mn - 2m -1}
}~=~2m+~1$.
\end{flushleft}}

\vspace{10pt}

\vspace{1pt}

\begin{flushleft}
(vi)\phantom{asdasdsadfdfdsf}$
\K{n =1}{\infty}
\displaystyle{\frac{f_{n}(n(n + 1)m + 1) + mn + m - 1 }
{f_{n}(n(n - 1)m + 1) + mn - m - 2}}
= 2m+ 1$.
\end{flushleft}

\vspace{10pt}

\begin{flushleft}
(vii)
Let $\displaystyle{ A_{n}/B_{n}}$ denote the convergents to the
continued fraction  below
and suppose   $\lim_{n \to
\infty}\displaystyle{n^{2}/B_{n}} = 0$ . Then

\vspace{1pt}

\begin{equation*}
\K{n =1}{\infty}
\displaystyle{\frac{f_{n}((n - 1)^{2}m + (m - 2)(n - 1) - 1) - m^{2}n + m - 1 }
{f_{n}((n - 2)^{2}m + (m - 2)(n - 2) - 1) - m^{2}(n - 2) + m - 2}}
= -1.
\end{equation*}
\end{flushleft}

\vspace{10pt}

\vspace{1pt}

\begin{flushleft}(viii)\phantom{asdasdsa}
$\K{n =1}{\infty}
\displaystyle{\frac{f_{n}(n(n + 1)m + 2n + 1) - (m^{2}(n + 1) + m + 1) }
{f_{n}(n(n - 1)m + 2n - 1) - (m^{2}(n - 1) + m + 2)}}
= 2m+ 3$.
\end{flushleft}

\vspace{10pt}

\begin{flushleft}
(ix) $\K{n =1}{\infty}
\displaystyle{\frac{
\begin{matrix}
-1 - 8\,m - 32\,m^2 - 128\,m^3 - 64\,m^2\,n +
\phantom{asddfdsfgsdfsdfsdfdgh} \\
 \phantom{asddffdfdsfgsdfsdfdgh}
\left( m + 16\,m^3 + \left( 1 + 16\,m^2 \right) \,n +
      4\,m\,n^2 \right) \,f_{n}
\end{matrix}}
{
\begin{matrix}
&-2 - 8\,m + 96\,m^2 - 128\,m^3 - 64\,m^2\,n +
\phantom{asddgfdgfdgsdfdsfggdsfsdh} \\
& \phantom{asdggh}   ( -1 + 5\,m - 16\,m^2 + 16\,m^3  +   \left( 1 - 8\,m + 16\,m^2 \right) \,n + 4\,m\,n^2
      ) \,f_{n}
\end{matrix}}}$
\end{flushleft}
\[=
\frac{1 + 5\,m + 16\,m^2 + 16\,m^3}{m + 16\,m^3}
\]
\end{corollary}

\begin{proof}
In each case below an easy check shows that with the given choices for
$ \{\,G_{n}\,\}_{n =-1}^{\infty}$,
$ \{\,a_{n}\,\}_{n = 1}^{\infty}$ and
$ \{\,b_{n}\,\}_{n = 1}^{\infty}$ that
equation~\eqref{E:hyp2}  holds, that the continued fraction in
question corresponds to the continued fraction
$\K{n = 1}{\infty}\displaystyle{\frac{s_{n}}{t_{n}}}$
of proposition~\eqref{T:hype},
 that if $\{A_{n}/B_{n}\}$ are the convergents to this continued
fraction then $\lim_{n \to \infty}\displaystyle{G_{n}/B_{n}} = 0$
 and that (by fact or assumption) no $s_{n}
= 0$.
Finally, the limit of the continued fraction is $G_{0}/G_{-1}$. The
fact that some early partial quotients may be negative does affect
any of the results - a tail of the continued fraction will have all
terms positive so that $\lim_{n \to \infty}\displaystyle{G_{n}/B_{n}}
= 0$ will hold for the tail which will then converge and the continued
fraction will then collapse from the bottom up to give the result.

Remark: In some cases the result holds if $f_{n}$ is a \emph{constant}
polynomial such that $f_{n} \geq 1$ for $n \geq 1$.

(i) Let $ G_{n} = mn + k$,
$ a_{n} = -1$,\,and \, $b_{n} = 2$.

\vspace{5pt}

(ii) Let $ G_{n} = n(n + 1)m + 1$,
$ a_{n} = mn  - 1$\,\, and \,\,  $b_{n} = -mn + (2m + 2) $.

\vspace{5pt}

(iii) Let $ G_{n} = (n + 2)^{2}$,   $ a_{n} = 4n + 5$ and $b_{n} = -4n + 4$ .

\vspace{5pt}

(iv) Let
$ G_{n} = m^{2}n^{2} + n(3m^{2} - 2m) + 2m^{2} - 2m+ 1$,
$a_{n}$$ = mn + m - 2$ and $b_{n} = -mn + m + 3$.

\vspace{5pt}

(v)   Let
$ G_{n} = n(n + 1)m^{2} + 2m(n + 1) + 1$,
$a_{n} = mn$ and $b_{n} = -mn + 2m + 1$.

\vspace{5pt}

(vi) Let
$ G_{n} = (n + 2)(n + 1)m  + 1$,
$a_{n} = mn + m - 1$ and $b_{n} = -mn + m + 2$.

\vspace{5pt}

(vii) Let
$ G_{n} = mn^{2} + (m - 2)n  - 1$,
$a_{n} = -m^{2}n + m - 1$ and $b_{n} = m^{2}n - 2m^{2} - m + 2$.

(viii) Let
$ G_{n} = (n + 2)(n + 1)m + 2n + 3$,
$a_{n} = -(m^{2}n + m^{2} + m + 1)$, and $b_{n} = m^{2}(n - 1) + m +
2$.

(ix) Let
$ G_{n}=1 + 5\,m + 16\,m^2 + 16\,m^3 +
   \left( 1 + 8\,m + 16\,m^2 \right) \,n + 4\,m\,n^2$,
$a_{n}=-1 - 8\,m - 32\,m^2 - 128\,m^3 - 64\,m^2\,n$ and
$b_{n} = 2 + 8\,m - 96\,m^2 + 128\,m^3 + 64\,m^2\,n$.

\end{proof}

Examples:\\
1) Letting $m = 5$ and $f_{n} = 10n^{8}$ in (ii) above gives

\vspace{5pt}

\[\K{n =
1}{\infty}\displaystyle{\frac{((n^{2} - n)5 + 1)10n^{8} + 5n - 1}
{((n^{2} - 3n + 2)5 + 1)10n^{8} + 5n - 12}} = 1.
\]\label{Ex:ex6}

\vspace{10pt}

2) Also in (ii), letting $f_{n}=n^{8}$ and $m$ be an arbitrary positive
  integer,
\[\K{n =1}{\infty}\displaystyle{\frac{((n^{2} - n)m + 1)n^{8} + nm-1}
{((n^{2} - 3n + 2)m + 1)n^{8} + mn - (2m + 2)}} = 1
.
\]\label{Ex:ex7}

\vspace{10pt}

3) Letting  $f_{n} = 2n^{5} $ and $ m = 3 $  in (vi) above gives
\eqref{Ex:eex10}
in the introduction. Similarly, letting $f_{n}=n^{10}$ gives
\eqref{Ex:eex8} in the introduction.

\vspace{10pt}

4) In (vii) above a general class of examples may be obtained by choosing
$m > 1$ and
$f_{n} > nm^{2}$ for $n \geq 1$. With the notation of the proposition
it can easily be seen that $s_{n}, t_{n} \geq 1 $ for $ n \geq 3$.
If $f_{n}$ is such that $B_{2}$ and $B_{3}$ are negative, then $B_{n}$
will be negative for all $n \geq 2$ and by a similar argument to the
reasoning behind condition (ii), it will follow that $\lim_{n \to
\infty}\displaystyle{n^{2}/B_{n}} = 0$ and the conditions of the
corollary will be satisfied. For example, letting $f_{n} = 16n$ and $m = 3$
gives that

\begin{align}
&\K{n =1}{\infty}\displaystyle
{\frac{48n^{3} - 80n^{2} + 7n +2}
{48n^{3} - 176n^{2} + 135n + 19}} = -1.\notag
\end{align}\label{Ex:ex12}

\vspace{5pt}

All the examples in the last corollary were derived from
solutions to equation~\eqref{E:hyp2} where $G_{n}$ had degree $2$.
Table \ref{Ta:t1} below gives several families of solutions to
equation \eqref{E:hyp2}, where $G_{n}$ is of degree 3 in $n$.
\vspace{2pt}

\begin{table}[ht]
  \begin{center}
	\begin{tabular}{| c | c | c | }
	\hline
	$G_{n}$ & $a_{n}$ & $b_{n}$  \\ \hline
	$(n^{2^{}} - 1)mn +1$ & $2mn(n-1) -1$ & $-2m(n^{2} -4n +3) +2$ \\ \hline
	$(2n^{2}+3n+1)mn +1$ & $-2n^{2}m(m-2) +m^{2}n $ &
               $-n^{2}(4m -2m^{2})$ \\
          {} & $+m-1$         &  $-n(7m^{2}-12m)$ \\
          {} &  {}            &       $+6m^{2}-7m+2$ \\ \hline
        $(n^{2}+3n+2)mn+1$    & $2n(n+1)m-1$ & $-2n(n-2)m+2$ \\ \hline
        $(n^{3}+6n^{2}+11n+6)m$ &$2mn^{2}+6mn$
                         & $-2mn^{2}+2m+2$ \\
		$+1$     &	$+4m-1$	&{}	\\\hline
	$mn^{3}+3mn^{2}$ & $-m^{2}n^{2}-n(m^{2}+m)$
		&$m^{2}n^{2}-n(2m^{2}-m)$ \\
        $+n(2m-3)-2$&$+m-1$	&$-4m+2$ \\\hline
	\end{tabular}
\phantom{asdf}\\
	\caption{Some infinite families of solutions to
        \eqref{E:hyp2} for $ G_{n}$ of degree 3.}\label{Ta:t1}
    \end{center}
\end{table}

Considering the third  and fourth row of entries in the
Table \ref{Ta:t1}, for example,
there is the following corollary to
\text{Proposition~\ref{T:hype}}:


\begin{corollary}\label{C:c12}
Let $f_{n}$ be a polynomial in
$n$ such that $f_{n}\geq 1$ for  $ n\geq 1$ and let $m$ be a
positive integer.

(i) If  $f(2) > 2$ then

\vspace{5pt}

\[
\K{n =1}{\infty}
\displaystyle{\frac{f_{n}((n^{2}-1)nm+1)+ 2mn(n+1)-1 }
{f_{n}((n^{2}-3n+2)mn +1) + 2mn(n-2) - 2}}
= 1.
\]

\vspace{5pt}

(ii)
\[
\K{n =1}{\infty}
\displaystyle{\frac{f_{n}((n^{2}+3n+2)nm+1)+ 2mn^{2}+6mn+4m-1 }
{f_{n}((n^{2}-1)nm+1) +2(n^{2}-1)m-2}}
= 6m+1.
\]

\end{corollary}

\vspace{2pt}
\begin{proof}
(i) In the light of the fact that $G_{n}$,\,$a_{n}$ and  $b_{n}$ satisfy
\eqref{E:hyp2} simply note that the numerator of the continued
fraction is $ f_{n}G_{n-1} + a_{n}$ and that the denominator is
$f_{n}G_{n-2} - b_{n}$. It is easily seen that $a_{n} \geq 1$, for all
$n \geq 1$ and that $b_{n} \geq 1$, for all
$n \geq 2$. It can also be shown that $B_{2}$ and $B_{3}$ are positive
for all $m$ and $f_{n}$ satisfying the conditions of the corollary.
In the light of what was said in an earlier remark this is sufficient
to ensure the result.

(ii) The proof of this follows  the same lines as that of (i) above.
\end{proof}

Taking  $f_{n}$ to be  $ n^{3} $ and $ m = 3 $ in part (i)
 gives \eqref{Ex:eex13}
in the introduction.

One could continue to prove similar results by finding other solutions
to equation~\eqref{E:hyp2} for degrees $2$ or $3$ or by going to
higher degrees, but these corollaries should be sufficient to
illustrate the principle at work.

\vspace{5pt}

\section{Infinite Polynomial Continued Fractions with Irrational Limits}
In this section we use a continued fraction-to-series
transformation equivalent to Euler's transformation to sum some
polynomial continued fractions with irrational limits.
\vspace{5pt}


\begin{theorem}\label{T:t4}
For  $N$ \,$\geq 1$
\begin{equation}\label{E:pn3}
b_{0} +
\K{n = 1}{N}\displaystyle{\frac{b_{n-1}x}{b_{n}-x}} =
\displaystyle{\frac{1}{\J{n = 0}{N}
\displaystyle{\frac{(-1)^{n}x^{n}}{\prod_{i = 0}^{n}
b_{i}}}}}.
\end{equation}
Thus, when $N \rightarrow \infty$,  the continued fraction converges
if and only if the series converges.
\end{theorem}
\begin{proof}
See, for example, Chrystal~\cite{gC22}, page 516, equation (14).
\end{proof}

\text{Remark:}
The irrationality criterion mentioned in the introduction means that if
$\{b_{n}\}_{i = 0}^{\infty}$ is a sequence of  integers, then
$\sum_{n = 0}^{\infty}
\displaystyle{(-1)^{n}x^{n}/b_{0}b_{1}\cdots
b_{n}}$ is not rational for $x = 1/m$, $m$  being a non-zero integer,
provided
 $|mb_{n}-1|\,\, \geq \,\,|mb_{n-1}| + 1$, for all $n$ sufficiently large.


{\allowdisplaybreaks
\begin{corollary}\label{C:3c2}
For all non-zero integers $m$ (and indeed for all non-zero real numbers $m$ )
\begin{align*}
&(i) \,\,\,\, \displaystyle{ 6m^{2} - 1 + \K{n = 1}{\infty}
\frac{m^{2}(4n^{2} + 2n)}{ m^{2}(4n^{2} + 10n + 6) -1}
= \frac{1}
   {\frac{1}{m} \csc(\frac{1}{m}) - 1 }}.
\end{align*}
\end{corollary}
}

\text{Remarks:}\\
(1) Glaisher, \cite {G74} states continued fraction expansions
essentially equivalent to this one and the one in the next corollary
.\\
(2) The irrationality criterion gives that $\sin (\frac{1}{m})$
is irrational for $m$ either a non-zero integers or the square-root of
a positive integer.

\begin{proof}
(i) In \text{Theorem~\ref{T:t4}}
let  $b_{n} = \displaystyle{(2n + 2)(2n + 3)m^{2}} $ and $x=1$. \\
\end{proof}


\begin{corollary}\label{C:3c3}
For all non-zero integers $m$ (and indeed for all non-zero real numbers $m$ )
\begin{align*}
&(i)\,  \displaystyle{ 2m^{2} + \K{n = 1}{\infty}
\frac{m^{2}(4n^{2} - 2n)}{ m^{2}(4n^{2} + 6n + 2) -1}
= \frac{1}
   {1 - \cos (\frac{1}{m}) }}.
\end{align*}
\end{corollary}

Note that the irrationality criterion gives that
$\cos (\frac{1}{m})$ is irrational for
$m$ either a non-zero integers or the square-root of
a positive integer.
\begin{proof}
(i) In \text{Theorem~\ref{T:t4}}
let  $b_{n} = \displaystyle{(2n + 1)(2n + 2)m^{2}}$\, and \,$x = 1$.\\
\end{proof}


{\allowdisplaybreaks
\begin{corollary}\label{C:3c4}
For all positive integers $\nu$ and all non-zero integers $m$ \,
(and indeed for all non-zero real numbers $m$ )
\begin{align*}
&(i) \, \displaystyle{ (\nu + 1) 4m^{2} + \K{n = 1}{\infty}
\frac{4m^{2}n(n + \nu )}{ 4m^{2}(n + 1)(n + \nu + 1) -1}
= \frac{1}
   {1 - (\nu)!(2m)^{\nu}J_{\nu}(\frac{1}{m}) }},
\end{align*}
 where $J_{\nu}(x)$ is the bessel function of the first kind
of order $\nu$.
\end{corollary}
}

The Tietze irrationality criterion shows that if $\nu$ is a
non-negative integer and $m$ is a non-zero integer or the squareroot of
a positive integer then
$J_{\nu}(\pm \frac{1}{m})$ is irrational.

\begin{proof}
(1) In \text{Theorem~\ref{T:t4}}
letting  $b_{n} = \displaystyle{4(n + 1)(\nu +n + 1)m^{2}}$\, and \,
$x = 1$ \, gives
{\allowdisplaybreaks
\begin{align*}
&4(\nu + 1)m^{2}
  +  \K{n = 1}{\infty}
       \displaystyle{\frac{n(\nu + n)4m^{2}}
             {(n+1)(n + \nu + 1)4m^{2} - 1}}\\
 &= \frac{1}{\J{n = 0}{\infty}
        \displaystyle{ \frac{(-1)^{n}(1/m)^{2n+2}}
{\prod_{i = 0}^{n}
4(i + 1)(\nu +i + 1)}}}
 =  \frac{1}{1 -\displaystyle{
           (\nu)!
                 (2m)^{\nu}J_{\nu}(1/m)}},
\end{align*}}
from the power series expansion for $J_{\nu}(x)$.
\end{proof}

Taking $\nu$ to be  $0$ and   $m = 2 $ gives \eqref{Ex:eex2}
in the introduction.
\eqref{Ex:eex3} in the introduction follows by letting
 $m = \frac{1}{\sqrt{2}} $ in Corollary ~\ref{C:3c2}.\label{Ex:ex16}


\allowdisplaybreaks{
\begin{corollary}\label{C:3c5}
For all non-zero integers $m$ (and indeed for all non-zero real numbers $m$)
\begin{multline*}
\displaystyle{1 + \cfrac{1}{6m^{3} - 1 +
        \displaystyle{  \K{n = 2}{\infty}
         \frac{(3n-5)(3n-4)(3n-3)m^{3}}{(3n-2)(3n-1)(3n)m^{3} -1}}}}\\
 = \left(\frac{1}{3}\exp{(\displaystyle{-1/m})}+
\frac{2}{3}\exp{(\displaystyle{1/2m})}
\cos{\left(\sqrt{3}/2m\right)}\right)^{-1}.
\end{multline*}
\end{corollary}
}
\begin{proof}
In \text{Theorem~\ref{T:t4}}
let \,$b_{0}=1$, $b_{n} = \displaystyle{(3n-2)(3n-1)(3n)}$,\,
for\, $n \geq 1$ and $x = 1/m^{3}$. Then
\[\displaystyle{1 + \cfrac{1/m^{3}}{6 - 1/m^{3} +
        \displaystyle{  \K{n = 2}{\infty}
         \frac{(3n-5)(3n-4)(3n-3)1/m^{3}}{(3n-2)(3n-1)(3n)  - 1/m^{3}}}}}
 = \frac{1}{\J{n = 0}{\infty}
        \displaystyle{ \frac{(-1)^{n}(1/m^{3n})}
{\displaystyle{(3n)!}}}}.
\]
Simplifying the continued fraction gives the left side
and finally the  right side equals
$ \left(\frac{1}{3}\exp{(\displaystyle{-1/m})}+
\frac{2}{3}\exp{(\displaystyle{1/2m})}
\cos{\left(\sqrt{3}/2m\right)}\right)^{-1}
$.
\end{proof}
By the Tietze criterion the irrationality of this last function follows when
 $m$ is a non-zero integer or the real cube-root of
a  non-zero integer.

\allowdisplaybreaks{

}
\end{document}